\documentclass{amsart}

\usepackage{layout} %introduit la gestion des marges
\usepackage[top=3cm, bottom=3cm, left=2cm, right=2cm]{geometry} %gestion des marges
\usepackage[english]{babel}
\usepackage[utf8]{inputenc}
\usepackage[T1]{fontenc}
\usepackage{amsmath}
\usepackage{amssymb}
\usepackage{mathrsfs}
\usepackage{amsthm}
\usepackage{enumerate} %(permet les énumérations)
\usepackage[colorlinks=true,urlcolor=blue, citecolor=red,linkcolor=blue]{hyperref}
\usepackage{enumitem}
\usepackage{comment}
\usepackage{mathtools}
\usepackage{esint}
\usepackage{faktor}
\usepackage{color}
\usepackage[capitalize]{cleveref}

\newtheorem{theorem}{Theorem}[section]

\newtheorem{corollary}[theorem]{Corollary}
\newtheorem{lemma}[theorem]{Lemma}
\newtheorem{remark}[theorem]{Remark}
\newtheorem{defi}[theorem]{Definition}

\newcommand{\scal}[2]{\left\langle #1,#2 \right\rangle}
\newcommand{\g}{\nabla}
\newcommand{\gp}{\nabla^\perp}
\newcommand{\di}{\mathrm{div}}
\newcommand{\lap}{\Delta}
\newcommand{\dr}{\partial}
\newcommand{\vol}{\mathrm{vol}}

\newcommand{\tr}{\mathrm{tr}}

\newcommand{\Conf}{\mathrm{Conf}}

\newcommand{\Imm}{\mathrm{Imm}}

\DeclareMathOperator{\osc}{osc}

\newcommand{\R}{\mathbb{R}}
\newcommand{\C}{\mathbb{C}}
\newcommand{\N}{\mathbb{N}}

\newcommand{\s}{\mathbb{S}}

\newcommand{\Sr}{\mathcal{S}}

\newcommand{\Ar}{\mathcal{A}}
\newcommand{\Er}{\mathcal{E}}
\newcommand{\Cr}{\mathcal{C}}

\newcommand{\Qr}{\mathcal{Q}}

\newcommand{\Arond}{\mathring{A}}

\newcommand{\inter}[2]{[\![#1,#2]\!]}
\newcommand{\ust}{\underset}
  
\newcommand{\aleq}{\lesssim}

\newcommand{\vp}{\varphi}
\newcommand{\ve}{\varepsilon}

\newcommand{\vn}{\vec{n}}

\title{Classification of branched Willmore spheres}
\author{Dorian Martino}
\address[Dorian Martino]{Institut de Mathématiques de Jussieu, Université Paris Cité, Bâtiment
	Sophie Germain, 75205 Paris Cedex 13, France}
\email{dorian.martino@imj-prg.fr}

\begin{document}
	
	\begin{abstract}
		Given a branched Willmore immersion from a closed Riemann surface, we show that Bryant's quartic is holomorphic. Consequently, this quartic vanishes when the underlying surface is a sphere and we obtain the full classification of branched Willmore spheres. To do so, we show that the asymptotic expansion in the $C^2$-topology of the conformal Gauss map at a branched point is a null straight line.
	\end{abstract}
	
	\maketitle
	
	\section{Introduction}
	
	Let $\Sigma$ be a closed Riemann surface. Given an immersion $\Phi : \Sigma\to \R^3$, its Willmore energy is given by
	\begin{align*}
		\Er(\Phi) &:= \int_\Sigma |\Arond_\Phi|^2_{g_\Phi} d\vol_{g_\Phi},
	\end{align*}
	where $\Arond_\Phi$ is the traceless part of the second fundamental form of $\Phi(\Sigma)$ and $g_\Phi := \Phi^*\xi$ is the metric induced by $\Phi(\Sigma)\subset (\R^3,\xi)$, with $\xi$ the flat metric on $\R^3$. An immersion $\Phi : \Sigma\to \R^3$ is said to be Willmore if it is a critical point of $\Er$. Thanks to the Gauss\textendash Bonnet formula, one could equivalently define Willmore surfaces as critical points of the following functional:
	\begin{align*}
		W(\Phi) = \int_{\Sigma} H^2_{\Phi}\, d\vol_{g_\Phi},
	\end{align*}
	where $H_\Phi = \frac{1}{2} \tr_{g_\Phi} A_{\Phi}$ is the mean curvature. When $\Sigma = \s^2$, Bryant \cite{bryant1984} proved that any smooth critical point of $\Er$ must be a minimal surface or the inversion of a minimal surface. To do so, he considers the conformal Gauss map $Y:\Sigma\to (\s^{3,1},\eta)$ of any immersion $\Phi : \Sigma\to \R^3$, see \cref{se:CGM} for a definition of the conformal Gauss map, where 
	\begin{align*}
		\eta := (dx^1)^2 + \cdots (dx^4)^2 - (dx^5)^2,&  & 	\s^{3,1} := \{x\in \R^5 : |x|^2_\eta = 1\}.
	\end{align*}
	He proved that $\Phi$ is a smooth Willmore immersion if and only if $Y$ is a smooth harmonic and conformal map. As a consequence, he shows that the quartic $\Qr := \scal{\dr^2_{zz} Y}{\dr^2_{zz} Y}_\eta (dz)^4$ is holomorphic and vanishes everywhere if and only if $\Phi$ is a conformal transformation of a minimal surface in $\R^3$. If $\Sigma = \s^2$, the only holomorphic quartic is $0$. Therefore, any smooth Willmore immersion of $\s^2$ must be a conformal transformation of a minimal surface.\\
	
	However when studying the compactness of Willmore immersions or the Willmore flow \cite{bernard2014,blatt2009,kuwert2012,laurain2018a,palmurella2022,riviere2020}, singularities arise. In particular, the so-called bubbles are, by construction, Willmore spheres with possible branch points, see \cref{se:CGM} for the definition of branch point and more precisely \Cref{def:branch_point} for the definition of a branched Willmore immersion. Another motivation to study branched Willmore spheres comes from the $16\pi$-conjecture originally formulated by Kusner, which states the following. The value of the min-max procedure between the round sphere and the round sphere with opposite orientation, should be the energy of a simply connected complete minimal surface with four planar ends. Furthermore, the optimal path should be realized by a Willmore flow. We refer to \cite[Open problem 3]{riviere2020}, \cite{palmurella2022} and to all the references therein for recent advances on the study of min-max procedures and the Willmore flow. When the underlying Riemann surface is $\s^2$, Rivière proved in \cite[Theorem I.1]{riviere2020} that min-max procedures for the Willmore energy are reached by branched Willmore spheres. In \Cref{cor:classification} below, we complete the classification of branched Willmore spheres. We prove that every branched Willmore sphere is a conformal transformation of a possibly branched minimal surface of $\R^3$. Hence, the energy of branched Willmore spheres is quantized, see for instance \cite[Sections 4-5]{bryant1984} or \cite[Corollary 2]{nguyen2012}, and we deduce that the values of min-max procedures are multiples of $4\pi$. We hope that the present work combined with the approach of Rivière \cite{riviere2020} and the works of \cite{hirsch2021,hirsch2023,michelat2020,michelat2021}, studying the Willmore index of inverted minimal surfaces, would permit to solve \cite[Open problem 3]{riviere2020}.\\
	
	Some cases of branched Willmore spheres have already been studied. Lamm\textendash Nguyen \cite{lamm2015} proved that Willmore spheres with at most three branch points counted with multiplicities are still conformal transformations of minimal surfaces. Their proof is based on the asymptotic analysis at branch points from Kuwert\textendash Schätzle \cite{kuwert2004,kuwert2007} and Bernard\textendash Rivière \cite{bernard2013,bernard2013a}. Michelat\textendash Rivière \cite{michelat2022} extended this classification to branched spheres arising as limit of smooth Willmore immersions. An alternative proof has been given by Bernard\textendash Laurain\textendash Marque, see \cite[Lemma 3.1]{bernard2023}. \\
	
	Thanks to the $\ve$-regularity given by \cite[Corollary 1.1]{bernard2023}, one can easily recover the known fact, see also \cite{lamm2015,michelat2022}, that the poles of $\Qr$ have at most order 2. By a blow up around $x\approx 0$, it holds
	\begin{align*}
		|x|^4|\scal{\dr^2_{zz}Y(x)}{ \dr^2_{zz} Y(x)}_\eta| &\aleq \| |\g Y|_\eta \|_{L^2 \left( B_{2|x|}\setminus B_{ \frac{|x|}{2} } \right)}^2 \aleq \|\g\vn\|_{L^2 \left( B_{2|x|}\setminus B_{ \frac{|x|}{2} } \right) }^2.
	\end{align*}
	Thanks to \cite[Proposition 1.2]{bernard2013a}, it holds $\g \vn\in L^p(B_1)$ for any $p>1$. Hence, we obtain that the following holds for any $\ve>0$, around the branch point:
	\begin{align*}
		|x|^4|\scal{\dr^2_{zz} Y(x)}{\dr^2_{zz} Y(x)}_\eta| &\aleq |x|^{2-\ve}.
	\end{align*}
	Thus, any pole of $\Qr$ must be of order at most 2. To obtain a refined estimate, again we proceed to a blow up on the branch point. Using the asymptotic expansion obtained in \cite[Theorem 1.8]{bernard2013a}, we show that the conformal Gauss map always goes to infinity  in a neighbourhood of a branch point. The author proved in \cite[Theorem 4.1]{martino2023} that up to an isometry of $\s^{3,1}$, the conformal Gauss map converges to a null geodesic in the $C^2$-topology. In particular, we now obtain freely a much better estimate of $\Qr$, namely given $z\approx 0$, there exists $\Theta_z\in \Conf(\R^3)$ and an annuli $A_z \subset B_{|z|}\setminus \{0\}$ of the form $A_z = B_{R_z}\setminus B_{r_z}$ and such that for any $x\in A_z $,
	\begin{align*}
		|x|^4|\scal{\dr^2_{zz} Y(x)}{\dr^2_{zz} Y(x)}_\eta| &\aleq \|\g \vn_{\Theta_z\circ \Phi} \|_{L^2 \left( B_{2|x|}\setminus B_{ \frac{|x|}{2} } \right)}^4.
	\end{align*}
	The main work of this paper is to estimate the right-hand side. To understand this estimate, we will understand the impact of $\Theta_z$ on the conformal Gauss map. As shown in \cite{martino2023}, the conformal transformation $\Theta_z$ corresponds to the multiplication by a matrix $M_z \in SO(4,1)$ on the conformal Gauss map such that $M_z Y$ is bounded in $B_{2|z|}\setminus B_{\frac{|z|}{2}}$. Since $M_z Y$ is bounded, we obtain a uniform bound on $H_{\Theta_z\circ \Phi}$ on the domain $A_z$. Therefore, we have reduced the question of estimating $\|\g \vn_{\Theta_z\circ \Phi} \|_{L^2 \left( B_{2|x|}\setminus B_{ |x|/2 } \right)}$ to the problem of estimating
	\begin{align*}
		\int_{B_{2|x|}\setminus B_{ \frac{|x|}{2} }} d\vol_{g_{\Theta_z\circ \Phi}} + \int_{B_{2|x|}\setminus B_{ \frac{|x|}{2} }} |\Arond_{\Theta_z\circ\Phi}|^2_{g_{\Theta_z\circ \Phi}} d\vol_{g_{\Theta_z\circ \Phi}}.
	\end{align*} 
	The last term is conformally invariant, hence it is independent of $\Theta_z$. To estimate the first term, we use the convergence of the conformal Gauss map to a geodesic to prove that it is controlled by $\| M_z \g Y \|_{L^2( B_{2|x|}\setminus B_{ |x|/2 } )}^2$. Using an adapted $\ve$-regularity estimate, we conclude that this term is bounded by $\|\g \vn_{\Phi} \|_{L^2 ( B_{2|x|}\setminus B_{ |x|/2 } )}^2$. Thus, we end up with an estimate of the following form:
	\begin{align*}
		|x|^4|\scal{\dr^2_{zz} Y(x)}{\dr^2_{zz} Y(x)}_\eta| &\aleq \|\g \vn \|_{L^2 \left( B_{2|x|}\setminus B_{ \frac{|x|}{2} } \right) }^4.
	\end{align*}
	Again by \cite[Proposition 1.2]{bernard2013a}, we conclude that there is no pole of order 1 or more:
	\begin{align*}
		|x||\scal{\dr^2_{zz} Y(x)}{\dr^2_{zz} Y(x)}_\eta| \ust{x\to 0}{=}o(1).
	\end{align*}
	In order to use the results of \cite{martino2023}, we need to have an immersion defined on a closed surface. Our main result is the following.	
	
	\begin{theorem}\label{th:Qbounded}
		Let $\Sigma$ be a closed Riemann surface. Consider $\Phi : \Sigma\to \R^3$ a branched Willmore immersion. Assume that its Gauss map $\vn$ satisfies 
		\begin{align*}
			\int_{\Sigma} |\g \vn|^2_g d\vol_g <\infty,
		\end{align*}
		where $g$ is the first fundamental form of $\Phi(\Sigma)\subset \R^3$. Then Bryant's quartic $\Qr=\scal{\dr^2_{zz} Y}{\dr^2_{zz} Y}(dz)^4$ is holomorphic on $\Sigma$.
	\end{theorem}
	
	If we consider $\Sigma=\s^2$, then $\Qr$ is a holomorphic quartic defined on $\s^2$. Hence $\Qr=0$ and we obtain the following result, which implies that the energy of branched Willmore spheres is quantized, see \cite[Corollary 2]{nguyen2012}. 
	
	\begin{corollary}\label{cor:classification}
		Any branched Willmore immersion of $\s^2$ into $\R^3$ is the inversion of a minimal surface.
	\end{corollary}
	
	From \cite[Theorem 10]{nguyen2012}, the Willmore energy of branched Willmore spheres is a multiple of $4\pi$. In particular, we have the following energy gap phenomenon for branched Willmore spheres. If $\Phi : \s^2\to \R^3$ is a branched Willmore sphere such that $W(\Phi) < 8\pi$, then $\Phi$ is a round sphere. We refer to \cite[Section 5]{nguyen2012} for various energy gap results.\\

	\textbf{Structure of the paper}\\
	In \cref{se:CGM}, we recall the definition of the geometric quantities we are working with. We recall the definition of a branch point and states the main estimates we will use. In \cref{se:proof}, we prove \cref{th:Qbounded}.\\
	
	\textbf{Acknowledgements}\\
	I would like to thank Paul Laurain for his constant support and advice. I also thank Tristan Rivière for important clarifications about the notion of branch point. Partial support through ANR BLADE-JC ANR-18-CE40-002 is acknowledged.
	
	\section{Branched Willmore immersion and conformal Gauss map}\label{se:CGM}
	
	In this section, we define first the geometric quantities we will work with. Second, we define the branched points and we recall asymptotic expansions of Willmore immersions around branch points. Third, we recall the definition of the conformal Gauss map.\\
	
	\textbf{Second fundamental form of an immersion}\\	
	Consider $\Phi :B_1\to \R^3$ a conformal immersion. Let $\lambda$ be the conformal factor : $g_\Phi := \Phi^*\xi = e^{2\lambda} h$, where $\xi$ is the flat metric on $\R^3$ and $h$ the flat metric on $B_1$. Let $\vn_\Phi$ be the Gauss map of $\Phi$, $A_\Phi$ its second fundamental form, $H_\Phi$ its mean curvature and $\Arond_\Phi$ the traceless part of the second fundamental form. They are given by the following formulas:
	\begin{align}
		(A_\Phi)_{ij} &:= -\scal{\dr_i \Phi}{\dr_j \vn_\Phi}_\xi, \nonumber\\
		H_\Phi &:= \frac{1}{2} \tr_g (A_\Phi) = \frac{e^{-2\lambda} }{2} \scal{\lap_h\Phi}{\vn_\Phi}_\xi, \nonumber\\
		\Arond_\Phi &= A_\Phi - H_\Phi g_\Phi, \nonumber\\
		\g \vn_\Phi &= -H_\Phi\g \Phi - e^{-2\lambda} \Arond_\Phi\g \Phi. \label{eq:gvn_H_Arond}
	\end{align}
	
	\textbf{Branched conformal immersions}	\\
	Let $\Sigma$ be a closed Riemann surface. In this work, a map $\Phi : \Sigma\to\R^3$ is called a branched immersion if there exists a finite set $\Sr\subset \Sigma$, called the set of branch points, such that $\Phi : \Sigma\setminus \Sr\to \R^3$ is a smooth immersion and 
	\begin{align*}
		\int_\Sigma |A_\Phi|^2_{g_\Phi} d\vol_{g_{\Phi}} <\infty.
	\end{align*}
	The above condition has to be understood has
	\begin{align*}
		\sup\left\{ \int_{U} |A_\Phi|^2_{g_\Phi} d\vol_{g_{\Phi}} : U\Subset \Sigma\setminus \Sr \right\}<\infty.
	\end{align*}
	See also Definition 2.2 in \cite{lamm2015} and the text below. Under this assumption, $g_\Phi$ is conformal to a smooth metric $h$ with constant curvature defined on the whole surface $\Sigma$. We sketch the proof of this result, it is detailed in \cite[p.332-334]{kuwert2004}. On each disk $D\subset \Sigma\setminus \Sr$ satisfying $\int_D |A_\Phi|^2_{g_\Phi} d\vol_{g_{\Phi}} <\frac{8\pi}{3}$, there exists conformal coordinates for $\Phi$, see for instance \cite[Theorem 5.1.1]{helein2002}. Thus, we can cover $\Sigma\setminus \Sr$ by charts $(U_j,\zeta_j)$ such that $\Phi\circ\zeta_j:B_1\to \R^3$ is conformal. Hence each $\zeta_k^{-1}\circ \zeta_j$ is conformal between two domains of $\R^2$. In particular, the family $(U_j,\zeta_j)$ generates a smooth conformal structure on $\Sigma\setminus \Sr$. Consider a given branch point $p\in\Sr$. Let $\Phi_p$ be the inversion of $\Phi$ with respect to the sphere centred at $p$ and radius $1$. Then $\Phi_p:\Sigma\to \R^3$ has a complete end at $p$ and a branch point at every point of $\Sr\setminus \{p\}$. Furthermore, $\Phi_p$ still satisfies 
	\begin{align*}
		\int_\Sigma |A_{\Phi_p}|^2_{g_{\Phi_p}} d\vol_{g_{\Phi_p}} <\infty.
	\end{align*}
	By Huber's theorem, see for instance \cite[Theorem 15]{huber1957} or \cite[Theorem 2.4.10]{saearp2015}, there exists a disk $D_p\subset \Sigma$ containing $p$ such that $\Phi_p(D_p)$ is conformally equivalent to $B_{\R^2}(0,1)\setminus \{0\}$ with the flat metric. Together with the charts $(U_j,\zeta_j)$ defined previously, we obtain a conformal structure on the whole surface $\Sigma$. Associated to this conformal structure, there is a smooth metric $h$ with constant curvature on $\Sigma$. Thus, $g_\Phi$ is conformal to $h$, with a conformal factor degenerating around the branch points.\\

	\textbf{Branched Willmore immersions}\\	
	Some precautions have to be taken in order to define branch Willmore surfaces. A first idea would be to define a branched conformal Willmore immersion as an immersion $\Phi : \Sigma\setminus \{p_1,\ldots,p_I\} \to \R^3$, where $p_1,\ldots,p_I$ are some given points in $\Sigma$, satisfying the following equation and energy bound:
	\begin{align}
		& \lap_g H + |\Arond|^2_g H = 0\ \ \ \text{ on }\Sigma\setminus \{p_1,\ldots,p_I\}, \label{eq:KS_branch}\\
		& \sup \left\{\int_U |A_\Phi|^2_{g_\Phi}\, d\vol_{g_\Phi} : U\Subset \Sigma\setminus \{p_1,\ldots,p_I\} \right\}<\infty. \label{eq:KS_energy_bound}
	\end{align}
	In order to shorten the notations, the condition \eqref{eq:KS_energy_bound} is sometimes written as $\|A_\Phi\|_{L^2(\Sigma)}<\infty$ or $\|A_\Phi\|_{L^2(\Sigma\setminus \{p_1,\ldots,p_I\})}<\infty$. Since no distributional meaning of $A_\Phi$ is needed across the branch point in order to use the analysis of Müller\textendash{\v S}ver{\' a}k \cite{muller1995}, the setting \eqref{eq:KS_branch}-\eqref{eq:KS_energy_bound} might seem a good setting at first sight. However, a careful reading of the works of Kuwert\textendash Schätzle \cite{kuwert2004,kuwert2007} and Lamm\textendash Nguyen \cite{lamm2015} shows that not only the condition \eqref{eq:KS_energy_bound} is used in the proofs of \cite{kuwert2004,kuwert2007,lamm2015}, but also the assumption that $A_\Phi$ has a distributional meaning across the singularities $p_1,\ldots,p_I$. Indeed, the proofs use that $A_\Phi \in L^2(\Sigma)$ in the sense of distributions, see for instance the proof of Lemma 4.1 in \cite{kuwert2004}, the end of the proof of Theorem I.1 at page 190 in \cite{kuwert2007} or the proof of Theorem 3.1 at p.176 in \cite{lamm2015}. We will call \textit{singular Willmore surfaces} the surfaces satisfying \eqref{eq:KS_branch}-\eqref{eq:KS_energy_bound}.\\
	
	 Even if Rivière emphasized the importance of the assumption $A_\Phi \in L^2(\Sigma)$ since 2008 in \cite{riviere2008}, through the notions of weak immersion and "true" branch point in his lecture notes \cite{riviere2016a} and all his works \cite{bernard2013,bernard2013a,bernard2014,riviere2013b,riviere2014,laurain2018a,riviere2020,palmurella2022}, the importance of this assumption is not written clearly in the existing literature. For instance, the work of Chen\textendash Li \cite{chen2014} uses the notion of singular Willmore surfaces, but quotes the work of Lamm\textendash Nguyen \cite{lamm2015} which uses the assumption $A_\Phi\in L^2(\Sigma)$, see below \Cref{rk:LN} and \Cref{rk:CL}. \\
	 
	 Indeed, the singularities arising in classical questions surrounding calculus of variations such as minimization among surfaces of given topological type or min-max problems enjoy this additional assumption, which leads to important consequences. In 2008, Rivière \cite{riviere2008} developed a weak formulation of the Willmore equation. He proved that \eqref{eq:KS_branch} is equivalent to the following equation:
	\begin{align*}
		\di\left(\g \vec{H} - 3\pi_{\vn}(\g \vec{H}) + *(\gp \vn\wedge \vec{H}) \right) = 0 \ \ \ \text{ on } \Sigma\setminus \{p_1,\ldots,p_I\},
	\end{align*}
	where $\vec{H} = H\vn$ and the operators $\wedge$ and $*$ are the wedge product and the Hodge star on vectors of $\R^3$. Rivière observed that the quantity $\left( \g \vec{H} - 3\pi_{\vn}(\g \vec{H}) + *(\gp \vn\wedge \vec{H}) \right)$ has a distributional meaning as long as $\vec{n}$ lies in $W^{1,2}(\Sigma;\s^2)$. \textit{A priori}, the formulation \eqref{eq:KS_branch} means that the quantity $\left( \g \vec{H} - 3\pi_{\vn}(\g \vec{H}) + *(\gp \vn\wedge \vec{H}) \right)$ defines a distribution on $\Sigma$ with support on the finite set $\{p_1,\ldots,p_I\}$. This implies that there exists some coefficients $a_{\alpha,i}\in \R$ such that for any smooth perturbation $\vec{w}\in C^\infty(\Sigma;\R^3)$, it holds:
	\begin{align}\label{eq:KS_singularities}
		\frac{d}{dt}W(\Phi + t\vec{w})_{|t=0} =  \delta W(\Phi)\cdot \vec{w} = \sum_{\alpha\in\N^2}\sum_{i=1}^I a_{\alpha,i} D^\alpha\vec{w}(p_i)\neq 0,
	\end{align}
	where the term $\delta W(\Phi)\cdot \vec{w}$ is the term we obtain in the weak formulation:
	\begin{align*}
		\delta W(\Phi)\cdot \vec{w} := \int_{\Sigma} \Big( \vec{H} \cdot (\lap \vec{w} - 3 \di(\pi_{\vn}(\g \vec{w}))) + *(\gp \vn \wedge \vec{H})\cdot \g \vec{w}\Big)\ d\vol_g.
	\end{align*}
	Hence, general singular Willmore surfaces do not satisfy the weak formulation of the Willmore equation across singularities. Therefore, one can wonder about surfaces satisfying the Willmore equation across the singularities:
	\begin{align}\label{eq:R_branch}
		\forall \vec{w}\in C^\infty(\Sigma;\R^3),\ \ \ \delta W(\Phi)\cdot \vec{w} = 0.
	\end{align}
	Branched conformal immersions satisfying \eqref{eq:R_branch} will be called \textit{branched conformal Willmore immersions}. Since this notion is compatible with weak limits of immersions, the singularities arising in classical variational problems (such as minimization among surfaces of given topological type, min-max problems, bubbling analysis, Willmore flow...) are given by branched Willmore immersions, see for instance \cite{bernard2014,riviere2016a,riviere2020,palmurella2022}. Such immersions satisfies additional regularity properties, see the work of Bernard\textendash Rivière \cite{bernard2013,bernard2013a} and the work of Kuwert\textendash Schätzle \cite{kuwert2004}, where for instance the assumption $d\vn\in L^2$ is needed to obtain estimates on the higher derivatives of the second fundamental form. Hence, the important distinction between \eqref{eq:KS_singularities} and \eqref{eq:R_branch} is the following. In \eqref{eq:R_branch}, the distributional derivative of $\vn$ is assumed to lie $L^2(\Sigma)$, providing the "initial estimate" across the singularities $\vn\in W^{1,2}(\Sigma)$. In \eqref{eq:KS_singularities}, some combinations of the distributional derivatives of $\vn$ is equal to a non-trivial linear combination of Dirac masses and their derivatives. Hence, even if \eqref{eq:KS_energy_bound} holds, the distributional derivative of $\vn$ does \textit{not} lie in $L^2(\Sigma)$ for general singular Willmore surfaces.\\
	
	We summarize this discussion in the following definition.
	\begin{defi}\label{def:branch_point}
		Let $(\Sigma,h)$ be a closed Riemann surface with a metric $h$ of constant curvature. A branched conformal Willmore surface $\Phi : \Sigma\to \R^3$ is a branched conformal immersion $\Phi : \Sigma\setminus \{p_1,\ldots,p_I\}\to \R^3$ for some points $p_1,\ldots,p_I\in \Sigma$ such that the Gauss map $\vn$ lies in $W^{1,2}(\Sigma)$ and such that \eqref{eq:R_branch} is satisfied.
	\end{defi}
	
	Given this definition, two remarks have to be considered in order to clarify the existing literature.
	
	\begin{remark}\label{rk:LN}
		In the work of Lamm\textendash Nguyen \cite{lamm2015}, the authors state a definition that might be understood as singular Willmore surfaces in Definition 2.4. However, a careful reading of the proof of Theorem 3.1 shows that they actually use the existence of a distributional derivative of $\vn$ across the singularities, hence they use the notion of branched Willmore surfaces. This can be seen in the statement of Theorem 2.3 in \cite{lamm2015}. The estimates on the higher derivatives of the second fundamental form is obtained only in the work of Bernard\textendash Rivière \cite{bernard2013,bernard2013a} and in Kuwert\textendash Schätzle \cite{kuwert2004}, where the assumption $A_{\Phi}\in L^2(\Sigma)$ is needed.
	\end{remark}
	
	\begin{remark}\label{rk:CL}
		Chen\textendash Li \cite{chen2014} proved that the energy of singular Willmore spheres is not quantized. Consequently, such surfaces cannot be conformal transformations of minimal surfaces. Actually, Theorem 2 in \cite{chen2014} shows that one can approximate the round sphere by singular Willmore surfaces with at least two singularities. Indeed, the authors introduce the following variational problem. Given an integer $m\geq 1$ and points $y_1\ldots, y_m\in\s^3$, we consider 
		\begin{align*}
			\beta_0(y_1,\ldots,y_m) := \inf \left\{ W(f) : f\in\Imm(\Sigma;\s^3),\ y_1\ldots,y_m\in f(\Sigma)
			\right\}.
		\end{align*}
		Chen\textendash Li claim in \cite[Theorem 2]{chen2014} to have obtained a branched Willmore immersion realizing $\beta_0(y_1,\ldots,y_m)$ if $\beta_0(y_1,\ldots,y_m)<8\pi$. However, if $f$ is a weak immersion realizing $\beta_0$, then $f$ must be singular in full generality because no variation is allowed at the points $y_1,\ldots,y_m$. Hence, the Euler\textendash Lagrange equation of $\beta_0$ must be of the form \eqref{eq:KS_singularities}. \\
		
		One can easily get an intuition as follows. Consider $\ve>0$ and the two points $y_1 = (\ve,0,0)$ and $y_2=(-\ve,0,0)$. Given $R>0$, let $C_R := \{0\}\times \{Rx : x\in\s^1\} \subset \R^3$ be the circle of radius $R$ centred at the origin and contained in the plane $\{0\}\times \R^2$. Using the ideas of \cite{chen2014}, one can obtain a solution to the following variational problem:
		\begin{align*}
			\beta := \inf \left\{ W(f) : f\in\Imm(B^2(0,R);\R^3),\ f(\dr B(0,R)) = C_R,\ y_1,y_2\in f(B(0,R))
			\right\}.
		\end{align*}
		Letting $R\to \infty$, we obtain a $W^{2,2}$-immersion $f_0 : \C \to \R^3$ such that $y_1,y_2\in f_0(\C)$, $f$ has density 1 at $y_1$ and $y_2$, $f_0$ is a smooth Willmore immersion on $\C\setminus f_0^{-1}(\{y_1,y_2\})$ and $f_0$ has one flat end. After inverting $f_0$, we obtain a weak immersion $f_1:\s^2 \to \R^3$ with exactly two singularities and satisfying the Willmore equation away from these singularities, the singularity of flat end is indeed removed by the inversion. If one could apply the work of Lamm\textendash Nguyen \cite{lamm2015}, as used in the proof of Theorem 2 in \cite{chen2014}, then $f_1$ would be a conformal transformation of a minimal surface. However, up to reducing $\ve$ in the definition of $y_1,y_2$, we obtain that $W(f_1)<4\pi+1$. Hence $f_1$ is the round sphere. Thus, $f_0$ must be a flat plane and we obtain a contradiction. Thus, \cite{lamm2015} can not be applied in this context. Hence the objects of \cite{chen2014} are not branched Willmore surfaces, but rather singular Willmore surfaces.
	\end{remark}
	
	We now explicit the asymptotic expansion of a branched Willmore immersion around a given branch point. Consider a conformal branched Willmore immersion $\Phi : B_1\setminus \{0\}\to \R^3$. By \cite[Lemma A.5]{riviere2014}, there exists $C_0>1$ and $\theta\in\N$, called the order or the multiplicity of the branch point, such that the following holds for any $x\in B_1\setminus \{0\}$:
	\begin{align}\label{asy:Phi_gPhi}
		C_0^{-1}|x|^{1+\theta} \leq |\Phi(x)| \leq C_0|x|^{1+\theta}, & & C_0^{-1}|x|^\theta \leq |\g \Phi(x)| \leq C_0|x|^\theta.
	\end{align}
	By \cite[Theorem 1.8]{bernard2013a}, the mean curvature vector $\vec{H}_\Phi$ of $\Phi$ satisfies the following asymptotic expansion. There exists $\alpha\in\inter{0}{\theta}$ and two vectors $\vec{E},\vec{\gamma}\in \C^3$ such that, in complex notations, it holds:
	\begin{align}\label{asy:H}
		\vec{H}_\Phi(z) &\ust{z\to 0}{=} \Re\left(\vec{E}z^{-\alpha} + \vec{\gamma} \log(|z|) \right)(1+o(1)).
	\end{align}
	
	\textbf{Conformal Gauss map for immersions in $\R^3$}\\
	For an introduction to conformal Gauss maps, see for instance \cite{hertrich-jeromin2003,marque2021}. By definition, the conformal Gauss map $Y$ of an immersion $\Phi:\Sigma\to\R^3$ is given by
	\begin{align}\label{def:CGM}
		Y := H_\Phi \begin{pmatrix}
			\vspace*{0.2em} \Phi\\ \vspace*{0.2em} \frac{|\Phi|^2-1}{2} \\ \vspace*{0.2em} \frac{|\Phi|^2+1}{2}
		\end{pmatrix} + \begin{pmatrix}
			\vn\\
			\scal{\vn}{\Phi}_\xi\\
			\scal{\vn}{\Phi}_\xi
		\end{pmatrix}.
	\end{align}
	In particular, it holds $H_\Phi = Y_5-Y_4$. A direct computation yields 
	\begin{align}\label{eq:gY}
		\g Y = (\g H_\Phi)\begin{pmatrix}
			\vspace*{0.2em} \Phi\\ \vspace*{0.2em} \frac{|\Phi|^2-1}{2} \\ \vspace*{0.2em} \frac{|\Phi|^2+1}{2}
		\end{pmatrix} - \Arond_\Phi \begin{pmatrix}
			\g \Phi\\ \scal{\Phi}{\g \Phi}_\xi\\ \scal{\Phi}{\g \Phi}_\xi
		\end{pmatrix}.
	\end{align}
	We have denoted $(\Arond_\Phi \g \Phi)_i = (\Arond_\Phi)_{ij}\g^j \Phi$, where the index is raised with the metric $g_\Phi$. Given any immersion $\Phi$, the conformal Gauss map is conformal to $\Phi$. Indeed it holds $Y^*\eta = \frac{|\Arond_\Phi|^2_{g_\Phi}}{2} g_\Phi$. Furthermore, $\Phi$ is Willmore if and $Y$ is harmonic.\\
	
	\textbf{Conformal Gauss map for immersions in $\s^3$}\\
	We will also use the representations in $\s^3$ in order to use the estimates from the proof of \cite[Theorem 4.1]{martino2023}. Given an immersion $\Psi : \Sigma\to\s^3$ with mean curvature $H_\Psi$ and Gauss map $N_\Psi$, the conformal Gauss map $Y$ of $\Psi$ is given by 
	\begin{align*}
		Y &= H_\Psi\begin{pmatrix}
			\Psi\\ 1 
		\end{pmatrix} + \begin{pmatrix}
			N_\Psi\\ 0
		\end{pmatrix}.
	\end{align*}
	The derivatives of $Y$ are given by
	\begin{align}\label{eq:gY_s3}
		\g Y &= (\g H_\Psi) \begin{pmatrix}
			\Psi\\ 1
		\end{pmatrix} - \Arond_\Psi\begin{pmatrix}
			\g \Psi\\ 0
		\end{pmatrix},
	\end{align}
	where $\Arond_\Psi$ is the second fundamental form of $\Psi$ and $(\Arond_\Psi\g \Psi)_i = (\Arond_\Psi)_{ij}\g^j \Psi$, with the index raised by the metric $\Psi^*\xi$. If $\pi : \s^3\setminus\{\text{north pole} \} \to \R^3$ is the stereographic projection, and $\Phi = \pi\circ\Psi$, then the two definitions of the conformal Gauss map of $\Phi$ and $\Psi$ coincide, see \cite[Equation (72)]{marque2021}:
	\begin{align}\label{eq:equality_CGM}
		H_\Phi \begin{pmatrix}
			\vspace*{0.2em} \Phi\\ \vspace*{0.2em} \frac{|\Phi|^2-1}{2} \\ \vspace*{0.2em} \frac{|\Phi|^2+1}{2}
		\end{pmatrix} + \begin{pmatrix}
			\vn\\ \scal{\vn }{\Phi}\\ \scal{\vn}{\Phi}
		\end{pmatrix} = H_\Psi \begin{pmatrix}
			\Psi\\ 1
		\end{pmatrix} + \begin{pmatrix}
			N_\Psi\\ 0
		\end{pmatrix}.
	\end{align}
	
	\section{Proof of \cref{th:Qbounded}}\label{se:proof}
	Until the end of the section, we fix $\Sigma$ a Riemann surface and $\Phi : \Sigma\to \R^3$ a  branched Willmore immersion. Let $\xi$ be the flat metric on $\R^3$. Consider a smooth metric $h$ on $\Sigma$ and let $\lambda :\Sigma\to \R$ be the conformal factor of $\Phi$, i.e. $\Phi^*\xi = e^{2\lambda}h$. Let $\Sr = \{p_1,\cdots, p_q\}$ the set of branched points of $\Phi$. In \cref{subsec:check_hyp}, we define a blow up around a given branch point and prove that our setting falls into a minor adaptation of the proof \cite[Theorem 4.1]{martino2023}. In \cref{subsec:Proof}, we conclude the proof of \cref{th:Qbounded}.\\
	
	We now focus on a unique branched point, that we assume to be at $0$. We denote $\theta\in\N$ the order of the branch point.
	
	\subsection{Blow up on a fixed branch point}\label{subsec:check_hyp}
	
	In this section, we proceed to a blow up around the origin. We will use complex notations. We prove that our problem falls into a minor adaptation of the setting of \cite[Theorem 4.1]{martino2023}.\\
	
	There are two main hypothesis in this theorem. First, the immersion has to be defined and smooth on a closed surface. Second, $\Phi$ has to satisfy some estimates on a given cylinder. The first hypothesis is a technical one needed only to prove the $L^\infty$-bound (23) of Lemma 5.4, see Remark 5.1 in \cite{martino2023}. We will show that the proof still holds when $\Phi$ has branch points, see \cref{lm:Good_ball}. To obtain the second hypothesis, we will proceed to a blow up around a given branch point, see \cref{lm:hyp_cvg}. \\
	Let $r_k >0$ be a sequence of radii converging toward $0$, and for any $z\in \C \setminus B_{r_k}$, let 
	\begin{align}\label{def:Phik}
		\Phi_k(z) := \Phi\left( \frac{r_k}{z} \right).
	\end{align}
	First, on the annuli $\C\setminus B_1$, we have the following estimates:
	\begin{lemma}\label{lm:hyp_cvg}
		Let $\ve_0$ be given by \cite[Theorem 4.1]{martino2023}. For $k$ large enough, it holds
		\begin{align*}
			\sup_{k\in\N}\int_{\C\setminus B_1} |A_{\Phi_k}|^2_{g_{\Phi_k}} d\vol_{g_{\Phi_k}} <\ve_0,
		\end{align*}
		and
		\begin{align*}
			\sup_{\rho\in \left[ 1,\infty \right)} \int_{B_{2\rho}\setminus B_\rho}  |A_{\Phi_k}|^2_{g_{\Phi_k}} d\vol_{g_{\Phi_k}} \xrightarrow[k\to\infty]{}{0}.
		\end{align*}
	\end{lemma}
	\begin{proof}
		Using \eqref{def:Phik} and the fact that the Willmore energy is a geometric quantity, it holds
		\begin{align*}
			\int_{\C\setminus B_1} |A_{\Phi_k}|^2_{g_{\Phi_k}} d\vol_{g_{\Phi_k}} =  \int_{B_{r_k}}|A_{\Phi}|^2_{g_{\Phi}} d\vol_{g_{\Phi}}.
		\end{align*}
		Since $\int_{B_1} |A_{\Phi}|^2_{g_{\Phi}} d\vol_{g_{\Phi}}<\infty$ and we integrate on a ball of radius converging to $0$, we conclude that for $k$ large enough, it holds
		\begin{align*}
			\int_{\C\setminus B_1} |A_{\Phi_k}|^2_{g_{\Phi_k}} d\vol_{g_{\Phi_k}} < \ve_0.
		\end{align*}
		Using the same argument on every dyadic annuli, we obtain:
		\begin{align*}
			\sup_{\rho\in \left[ 1,\infty \right)} \int_{B_{2\rho}\setminus B_\rho}  |A_{\Phi_k}|^2_{g_{\Phi_k}} d\vol_{g_{\Phi_k}} &= \sup_{\rho\in \left( 0,\frac{r_k}{2} \right]} \int_{B_{2\rho}\setminus B_\rho}  |A_{ \Phi}|^2_{g_{\Phi}} d\vol_{g_{\Phi}} \xrightarrow[k\to\infty]{}{0}.
		\end{align*}
	\end{proof}
	
	The sequence $(\Phi_k)_{k\in\N}$ is defined on the closed surface $\Sigma$. Furthermore, since $\Phi_k$ is only a change of parametrization of $\Phi$, we obtain a uniform bound on the $L^2$-norm of its second fundamental form: it holds
	\begin{align*}
		\int_\Sigma |A_{\Phi_k}|^2_{g_{\Phi_k}} d\vol_{g_{\Phi_k}} = \int_\Sigma |A_\Phi|^2_{g_\Phi} d\vol_{g_\Phi}.
	\end{align*}
	As explained previously, in order to apply all the estimates on the conformal Gauss map coming from \cite[Section 5]{martino2023}, we need to check the bound (23) of Lemma 5.5 in \cite{martino2023}, we refer again to Remark 5.1 in \cite{martino2023}. The proof of this bound is a direct application of \cite[Lemma 5.13]{riviere2016a}. We prove in \cref{lm:Good_ball} that this lemma still holds for branched immersions. \cref{lm:Good_ball} follows from direct consequences of the existence of density and consequences of the monotonicity formula given by \cite[Lemma 5.6]{riviere2016a}. The existence of a density follows from  \cite[Theorem 3.1]{kuwert2012b}. We now check that the monotonicity formula for smooth immersions still holds for branched immersions. 
	
	\begin{lemma}[Monotonicity formula for branched immersions]
		Let $\Psi : \Sigma \to \R^3$ be a branched immersion such that 
		\begin{align*}
			\int_\Sigma |A_\Psi|^2_{g_\Psi} d\vol_{g_\Psi} <\infty.
		\end{align*}
		Let $M = \Phi(\Sigma)$. Then, for any $x_0\in\R^3$ and $0<t<T$, it holds
		\begin{align*}
			& T^{-2} \mathrm{Area}\left[ M \cap B(x_0,T)  \right] - t^{-2} \mathrm{Area}\left[ M \cap B(x_0,t)  \right]\\
			\geq &-\frac{1}{4} \int_{M \cap B(x_0,T) \setminus B(x_0,t)} H^2_\Psi d\vol_{g_\Psi}  - \frac{1}{T^2} \int_{M\cap B(x_0,T)} \scal{x-x_0}{\vec{H}_\Psi} d\vol_{g_\Psi} + \frac{1}{t^2}\int_{M\cap B(x_0,t)} \scal{x-x_0}{\vec{H}_\Psi} d\vol_{g_\Psi}.
		\end{align*}
	\end{lemma}
	
	\begin{proof}
		We start from the monotonicity formula for immersions with boundary coming from \cite[Lemma A.3]{riviere2013b}. Consider $\{p_1,...,p_q\}$ the branch points of $\Psi$. Let $h$ be a metric of constant curvature conformal to $g_\Psi$ and $\lambda$ be the conformal factor: $g_\Psi = e^{2\lambda}h$. Let $\delta>0$ and consider the restriction $\Psi : \Sigma \setminus \bigcup_{i=1}^q B_h(p_i,\delta) \to \R^3$. Let $M_\delta := \Psi\left(\Sigma \setminus \bigcup_{i=1}^q B_h(p_i,\delta) \right)$. Then for any ball $B(x_0,T)\subset \R^3$ and $0<t<T$, it holds
		\begin{align*}
			& T^{-2} \mathrm{Area}\left[ M_\delta \cap B(x_0,T)  \right] - t^{-2} \mathrm{Area}\left[ M_\delta \cap B(x_0,t)  \right]\\
			\geq &-\frac{1}{4} \int_{M_\delta \cap B(x_0,T) \setminus B(x_0,t)} H^2_\Psi d\vol_{g_\Psi}  \\
			&- \frac{1}{T^2} \int_{M_\delta\cap B(x_0,T)} \scal{x-x_0}{\vec{H}_\Psi} d\vol_{g_\Psi} + \frac{1}{t^2}\int_{M_\delta\cap B(x_0,t)} \scal{x-x_0}{\vec{H}_\Psi} d\vol_{g_\Psi} \\
			&+ \frac{1}{2} \int_{\dr M_\delta \cap B(x_0,T)} \left( \frac{1}{T^2} - \frac{1}{\max(|x-x_0|,t)^2} \right) \scal{x-x_0}{\vn_\Psi} d\ell_{\dr M_\delta}.
		\end{align*}
		On the last line, the measure $d\ell_{\dr M_\delta}$ is the restriction of the volume form $d\vol_{g_{\Psi}}$ to the boundary $\dr M_\delta$. To recover the formula from \cite[Lemma 5.6]{riviere2016a}, we consider the limit $\delta\to 0$ and prove that the last term (the only boundary term) converges toward $0$. Coming back to the parametrization $\Psi$, it holds
		\begin{align*}
			& \left| \int_{\dr M_\delta \cap B(x_0,T)} \left( \frac{1}{T^2} - \frac{1}{\max(|x-x_0|,t)^2} \right) \scal{x-x_0}{\vn_\Psi} d\ell_{\dr M_\delta} \right|  \\
			\leq & \left( \frac{1}{T^2} + \frac{1}{t^2} \right) \int_{\bigcup_{i=1}^q \dr B(p_i, \delta)} |\scal{\Psi - x_0}{\vn_\Psi}| e^{\lambda} d\vol_h.
		\end{align*}
		Using the asymptotic expansion \eqref{asy:Phi_gPhi}, we obtain that $\Psi$ and $e^\lambda$ are bounded around a branch point. Since the measure of the circle $\dr B(p_i,\delta)$ converges to $0$, we obtain
		\begin{align*}
			\int_{\bigcup_{i=1}^q \dr B(p_i, \delta)} |\scal{\Psi - x_0}{\vn_\Psi}| e^{\lambda} d\vol_h \xrightarrow[\delta\to 0]{}{0}.
		\end{align*}
		Hence we recover the formula from \cite[Lemma 5.6]{riviere2016a} for branched immersions.
	\end{proof}
	
	By following verbatim the proof of \cite[Lemma 5.13]{riviere2016a}, we obtain:
	
	\begin{lemma}\label{lm:Good_ball}
		Let $\Psi_k : \Sigma \to \R^3$ be a sequence of branched immersions satisfying 
		\begin{align*}
			\sup_{k\in\N} \int_\Sigma |A_{\Psi_k}|^2_{g_{\Psi_k}} d\vol_{g_{\Psi_k}} <\infty.
		\end{align*}
		Then, there exists a ball $B\subset B_{\R^3}(0,1)$ such that up to a subsequence, the following holds for any $k\in\N$: 
		$$
		\Psi_k(\Sigma)\cap B = \emptyset.
		$$
	\end{lemma}
	
	Let $Y_k :\Sigma \to \s^{3,1}$ be the conformal Gauss map of $\Phi_k$. On the annuli $\C\setminus B_{r_k}$, it holds 
	\begin{align}\label{eq:CGM_Phik}
		Y_k(z) = Y\left(\frac{r_k}{z} \right).
	\end{align}
	
	\subsection{End of the proof of \cref{th:Qbounded}}\label{subsec:Proof}
	In this section, we conclude the proof of \cref{th:Qbounded}. First we show that around branch points, the conformal Gauss map always blow up, see \cref{lm:existence_oscillations_bp}. Therefore up to a conformal transformation of $\Phi_k$, we obtain the convergence of the conformal Gauss maps to a geodesic on a well chosen annuli, thanks to \cite[Theorem 4.1]{martino2023}. By \cref{lm:estimate_gvn}, we control the $L^2$-norm of the second fundamental form after the conformal transformation, so we conclude.\\
	
	We start by showing that the conformal Gauss map always blows up at branch points.
	\begin{lemma}\label{lm:existence_oscillations_bp}
		For any $M\in SO(4,1)$ and $s_1\in(0,1)$, it holds
		\begin{align*}
			\osc_{B_{s_1}\setminus B_s} (MY) := \sup_{x,y\in B_{s_1}\setminus B_s} |MY(x)-MY(y)|_\xi \xrightarrow[s\to 0]{}{ +\infty}.
		\end{align*}
	\end{lemma}
	\begin{proof}
		Thanks to the asymptotic expansions \eqref{asy:Phi_gPhi}-\eqref{asy:H} and the definition of a conformal Gauss map \eqref{def:CGM}, we have the following behaviour for $Y$ around $0$. There exists $\alpha\in\inter{0}{\theta}$, $\vec{E}\in\C^5$ and constants $a,b\in\C$ such that, in complex notations, it holds
		\begin{align*}
			Y(z) \ust{z\to 0}{=} \Re\left( \left( a z^{-\alpha} + b \log(|z|)\right) \vec{E} \right)\left( 1 + o(1) \right).
		\end{align*}
		In particular, for any matrix $M\in SO(4,1)$, it holds
		\begin{align*}
			MY(z)\ust{z\to 0}{=} \Re\left( \left( a z^{-\alpha} + b \log(|z|)\right) M\vec{E}\right)\left( 1 + o(1) \right).
		\end{align*}
		Therefore, given any $s_1\in(0,1)$, it holds
		\begin{align*}
			\osc_{B_{s_1}\setminus\{0\}} (MY) = +\infty.
		\end{align*}
	\end{proof}
	
	In order to use the estimates coming from the proof of \cite[Theorem 4.1]{martino2023}, we will use the parametrization on a cylinder. Using the map $\Big( (t,\vp)\in[0,\infty)\times\s^1 \mapsto e^{t+i\vp} \in \C\setminus B_1 \Big) $, we obtain a sequence $\Phi_k : [0,\infty)\times\s^1 \to \R^3$. Thanks to \cref{lm:hyp_cvg}, all the assumptions \cite[Theorem 4.1]{martino2023} are satisfied. Let $(\Theta_k)_{k\in\N}\subset \Conf(\R^3)$ and $(M_k)_{k\in\N}\subset SO(3,1)$ be the resulting conformal transformation and the associated matrix at the scale $0$. Given $k\in\N$, by \eqref{eq:CGM_Phik} and \cref{lm:existence_oscillations_bp} with the choices $M=M_k$ and $s_1=r_k$,  there exists $\kappa_k >0$ such that
	\begin{align}\label{hyp:construction_sk}
		\osc_{[0,\kappa_k]\times\s^1} M_k Y_k = 1.
	\end{align} 
	Applying \cite[Theorem 4.1]{martino2023}, we conclude that $M_k Y_k$ converges to a light-like geodesic in the $C^2$-topology. We now summarize all the estimates we will use. First Claim 5.8 of \cite{martino2023} yields the existence of a cylinder 
	\begin{align}\label{def:choice_Ark}
		\Ar_k \subset \Cr_k := [0,\kappa_k]\times \s^1,
	\end{align}
	such that 
	\begin{align}
		& \osc_{\Ar_k} M_k Y_k \geq \frac{1}{2}, \label{eq:oscillations_Ak}\\
		& \|M_k Y_k \|_{L^\infty(\Ar_k)} \leq C. \label{eq:Linfty_Y}
	\end{align}
	Then, from Lemma 5.18 and Remark 5.16 in \cite{martino2023}, the following convergence holds: 
	\begin{align}
		& \left\| \frac{|\g^2 (M_kY_k)|_\xi}{|\g(M_k Y_k)|^2_\xi} \right\|_{L^\infty(\Ar_k)} \xrightarrow[k\to\infty]{}{0}, \label{eq:hessianY_k}\\
		& \left| \frac{ \inf_{x\in\Ar_k} |\g(M_k Y_k)(x)|_\xi }{ \sup_{y\in \Ar_k} |\g(M_k Y_k)(y)|_\xi } -1 \right| \xrightarrow[k\to\infty]{}{0}, \label{eq:harnack_Y}\\
		& \left\| \frac{ |\g Y_k|_\eta }{ |\g (M_k Y_k)|_\xi } \right\|_{L^\infty(\Ar_k)} \xrightarrow[k\to\infty]{}{0}. \label{eq:conv_lightlike}
	\end{align}
	Finally, by the estimates (22) and (23) in Lemma 5.5 from \cite{martino2023} it holds
	\begin{align}
		\forall (s_k,\vp_k)\in \Ar_k,\ \ \ &\|\g(M_k Y_k)\|_{L^\infty([s_k-1,s_k+1]\times\s^1)} \leq C\|\g \vn_{\Theta_k\circ \Phi_k} \|_{L^2([s_k-2,s_k+2]\times\s^1) }, \label{eq:ve_regu}\\
		& \sup_{k\in\N} \|\Theta_k\circ \Phi_k\|_{L^\infty(\Ar_k)} <\infty. \label{eq:Linfty_Phi}
	\end{align}
	From \eqref{eq:CGM_Phik} and \eqref{eq:hessianY_k}, we obtain a first estimate on $\scal{\dr^2_{zz} Y}{\dr^2_{zz} Y}_\eta$. Since $Y_k$ is conformal, it holds $\scal{\dr_z Y_k}{\dr_z Y_k}_\eta = 0$ and $\scal{\dr^2_{zz}Y_k}{\dr_z Y_k}_\eta=0$. Hence for $k$ large enough and any $(s_k,\vp_k) \in \Ar_k$, it holds
	\begin{align}
		\left( r_k e^{-s_k}\right)^4 \left|\scal{\dr^2_{zz} Y}{\dr^2_{zz} Y}_\eta\left(r_k e^{-s_k-i\vp_k} \right)\right| &= |\scal{\dr^2_{zz} Y_k}{\dr^2_{zz} Y_k}_\eta (s_k,\vp_k)| \nonumber\\
		&\leq |\g^2 M_kY_k(s_k,\vp_k)|^2_\xi \nonumber\\
		&\leq |\g (M_k Y_k)(s_k,\vp_k)|^4_\xi. \label{eq:estimate_Q}
	\end{align}
	By the $\ve$-regularity \eqref{eq:ve_regu}, we obtain
	\begin{align}\label{eq:ve_regularity}
		|\g (M_k Y_k)(s_k,\vp_k)|^4_\xi &\leq C\left( \int_{[s_k-1,s_k+1]\times\s^1} |\g \vn_{\Theta_k\circ \Phi_k}|^2_{g_{\Theta_k\circ \Phi_k}} d\vol_{g_{\Theta_k\circ \Phi_k}} \right)^2.
	\end{align}
	
	\begin{lemma}\label{lm:estimate_gvn}
		There exists $C>0$ satisfying the following. For any $\ve>0$, there exists $k_0\in\N$ depending on $\ve$, such that for every $k\geq k_0$ and $(s_k,\vp_k)\in\Ar_k$:
		\begin{align*}
			\int_{[s_k-1,s_k+1]\times\s^1} |\g \vn_{\Theta_k\circ \Phi_k}|^2_{g_{\Theta_k\circ \Phi_k}} d\vol_{g_{\Theta_k\circ \Phi_k}} &\leq \ve |\g (M_k Y_k)(s_k,\vp_k)|^2_\xi + \int_{[s_k-1,s_k+1]\times\s^1} |\Arond_{\Phi_k}|^2_{g_{\Phi_k}} d\vol_{g_{\Phi_k}}.
		\end{align*}
	\end{lemma}
	
	\begin{proof}
		We consider the representation of the Willmore surfaces in $\s^3$. Consider $\omega : \R^3\to \s^3 \setminus\{\text{north pole}\}$ the inverse of the stereographic projection. Let $(s_k,\vp_k)\in\Ar_k$. Using \eqref{eq:gvn_H_Arond} and \eqref{eq:Linfty_Y} (which provides a bound on the mean curvature), it holds 
		\begin{align}
			\int_{[s_k-1,s_k+1]\times\s^1} |\g \vn_{\Theta_k\circ \Phi_k}|^2_{g_{\Theta_k\circ \Phi_k}} d\vol_{g_{\Theta_k\circ \Phi_k}} &= \int_{[s_k-1,s_k+1]\times\s^1} 2H^2_{\Theta_k\circ \Phi_k} + |\Arond_{\Theta_k\circ \Phi_k}|^2_{g_{\Theta_k\circ \Phi_k}} d\vol_{g_{\Theta_k\circ \Phi_k}} \nonumber\\
			&\leq C\int_{[s_k-1,s_k+1]\times\s^1} d\vol_{g_{\Theta_k\circ \Phi_k}} + \int_{[s_k-1,s_k+1]\times\s^1} |\Arond_{\Phi_k}|^2_{g_{\Phi_k}} d\vol_{g_{\Phi_k}}. \label{eq:estimate_gn}
		\end{align}
		We focus on the area term. By definition of $\omega$, it holds
		\begin{align*}
			g_{\omega\circ\Theta_k\circ \Phi_k} = \frac{4}{(1+|\Theta_k\circ\Phi_k|^2)^2} g_{\Theta_k\circ\Phi_k}.
		\end{align*}
		By \eqref{eq:Linfty_Phi}, we obtain
		\begin{align}\label{eq:areaPhi}
			\int_{[s_k-1,s_k+1]\times\s^1} d\vol_{g_{\Theta_k\circ \Phi_k}} &\leq C \int_{[s_k-1,s_k+1]\times\s^1} d\vol_{g_{\omega\circ\Theta_k\circ \Phi_k}}.
		\end{align}
		Now we show that the metric $g_{\omega\circ\Theta_k\circ \Phi_k}$ is controlled by $Y_k$. Again by \eqref{eq:conv_lightlike} and \eqref{eq:gY_s3}, it holds
		\begin{align*}
			\dr_t (M_k Y_k) \ust{k\to\infty}{=} (\dr_t H_{\omega\circ\Theta_k\circ\Phi_k}) \nu_k + o(\dr_t H_{\omega\circ\Theta_k\circ\Phi_k}),
		\end{align*}
		where 
		\begin{align*}
			\nu_k := \begin{pmatrix}
				\omega\circ\Theta_k\circ \Phi_k\\ 1
			\end{pmatrix}.
		\end{align*}
		Thus, if we consider the change of variable $s=s(t)$ given by $ds = \left( \fint_{\s^1} |\g (M_k Y_k)|^2_\xi d\vp\right)^\frac{1}{2} dt$, we conclude that
		\begin{align*}
			\dr_s (M_k Y_k) \ust{k\to\infty}{=} \frac{\nu_k}{\sqrt{2}} + o(1).
		\end{align*}
		Therefore, the metric $g_{\omega\circ\Theta_k\circ \Phi_k} = \nu_k^*\eta$ is bounded by above and below by the metric $[\dr_s (M_k Y_k)]^*\eta$. Using the Cauchy\textendash Schwarz inequality for the Euclidean metric, we obtain:
		\begin{align*}
			|\scal{\dr_i (M_k Y_k)}{\dr_j (M_k Y_k)}_\eta| \leq |\dr_i (M_k Y_k)|_\xi |\dr_j(M_k Y_k)|_\xi.
		\end{align*}
		Hence, we deduce that:
		\begin{align}
			\int_{[s_k-1,s_k+1]\times\s^1} d\vol_{g_{\pi\circ\Theta_k\circ \Phi_k}} &\leq C\int_{[s_k-1,s_k+1]\times\s^1} \sqrt{\det \left( \scal{\dr_i (M_k Y_k)}{\dr_j (M_k Y_k)}_\eta\right)_{1\leq i,j\leq 2} } dt d\vp \nonumber \\
			&\leq C\int_{[s_k-1,s_k+1]\times\s^1}|\g (\dr_s M_k Y_k)|_\xi^2 dt d\vp. \label{eq:estimate_Area_ThetakPhik}
		\end{align}
		We estimate the right-hand side, using $\dr_s = \left( \fint_{\s^1} |\g (M_k Y_k)|^2_\xi d\vp\right)^{-\frac{1}{2}} \dr_t$:
		\begin{align*}
			|\g (\dr_s M_k Y_k)|_\xi &= \left| \g\left( \frac{\dr_t (M_k Y_k)}{\left( \fint_{\s^1} |\g (M_k Y_k)|^2_\xi d\vp \right)^{1/2} } \right) \right|_\xi \\ 
			&\leq \frac{|\g^2 (M_k Y_k)|_\xi}{ \left( \fint_{\s^1} |\g (M_k Y_k)|^2_\xi d\vp \right)^{1/2} } + \frac{|\dr_t (M_k Y_k)|_\xi }{ \left( \fint_{\s^1} |\g (M_k Y_k)|^2_\xi d\vp \right)^{3/2} } \left| \fint_{\s^1} \scal{\g(M_k Y_k)}{\g^2 (M_k Y_k)}_\xi d\vp \right|. \\
		\end{align*}
		By the Cauchy\textendash Schwarz inequality, we conclude that 
		\begin{align}
			|\g (\dr_s M_k Y_k)|_\xi &\leq \frac{|\g^2 (M_k Y_k)|_\xi}{ \left( \fint_{\s^1} |\g (M_k Y_k)|^2_\xi d\vp \right)^{1/2} } + \frac{ |\g(M_k Y_k)|_\xi }{ \fint_{\s^1} |\g (M_k Y_k)|^2_\xi d\vp  } \left( \fint_{\s^1} |\g^2 (M_k Y_k)|_\xi^2 d\vp \right)^\frac{1}{2}. \label{eq:estimate_gdsY}
		\end{align}
		Plugging \eqref{eq:harnack_Y} into \eqref{eq:estimate_gdsY} and \eqref{eq:estimate_Area_ThetakPhik}, we obtain
		\begin{align*}
			\int_{[s_k-1,s_k+1]\times\s^1} d\vol_{g_{\omega\circ\Theta_k\circ \Phi_k}} &\leq \frac{C}{ \fint_{\{s_k\}\times \s^1} |\g (M_k Y_k)|^2_\xi d\vp } \int_{[s_k-1,s_k+1]\times\s^1}|\g^2 M_k Y_k|_\xi^2 dt d\vp.
		\end{align*}
		By \eqref{eq:hessianY_k}, it holds
		\begin{align*}
			\int_{[s_k-1,s_k+1]\times\s^1} d\vol_{g_{\omega\circ\Theta_k\circ \Phi_k}} &\ust{k\to\infty}{=}o\left( \frac{1}{ \fint_{\{s_k\}\times \s^1} |\g (M_k Y_k)|^2_\xi d\vp } \int_{[s_k-1,s_k+1]\times\s^1}|\g M_k Y_k|_\xi^4 dt d\vp \right).
		\end{align*}
		Using again \eqref{eq:harnack_Y}, we conclude that for any $\vp_k\in\s^1$:
		\begin{align*}
			\int_{[s_k-1,s_k+1]\times\s^1} d\vol_{g_{\omega\circ\Theta_k\circ \Phi_k}} &\ust{k\to\infty}{=}o\left( \int_{[s_k-1,s_k+1]\times\s^1}|\g M_k Y_k|_\xi^2 dtd\vp \right) \\
			&\ust{k\to\infty}{=} o\left( |\g M_k Y_k(s_k,\vp_k)|_\xi^2 \right).
		\end{align*}
		We conclude by plugging this estimate into \eqref{eq:areaPhi} and \eqref{eq:estimate_gn}.
	\end{proof}
	
	We now finish the estimate of \eqref{eq:estimate_Q}. By \cref{lm:estimate_gvn} and \eqref{eq:ve_regularity}, we obtain for $k$ large enough,
	\begin{align*}
		|\g(M_k Y_k)(s_k,\vp_k)|^4_\xi &\leq C\left( \int_{[s_k-1,s_k+1]\times\s^1} |\Arond_{\Phi_k}|^2_{g_{\Phi_k}} d\vol_{g_{\Phi_k}} \right)^2 \\
		&\leq C\left( \int_{[s_k-1,s_k+1]\times\s^1} |A_{\Phi_k}|^2_{g_{\Phi_k}} d\vol_{g_{\Phi_k}} \right)^2.
	\end{align*}
	Since the right-hand side is a geometric quantity, we can change the parametrization of the integral. We come back to the parametrization on an annuli:
	\begin{align*}
		|\g(M_k Y_k)(s_k,\vp_k)|^4_\xi &\leq C\left( \int_{B_{2e^{s_k}}\setminus B_{e^{s_k}/2}} |A_{\Phi_k}|^2_{g_{\Phi_k}} d\vol_{g_{\Phi_k}} \right)^2.
	\end{align*}
	Since $\Phi_k$ is a change of parametrization of $\Phi$, we obtain
	\begin{align*}
		|\g (M_k Y_k)(s_k,\vp_k)|^4_\xi &\leq C\left( \int_{B_{2r_k e^{-s_k}}\setminus B_{r_ke^{-s_k}/2}} |A_{\Phi}|^2_{g_\Phi} d\vol_{g_{\Phi}} \right)^2\\
		&\leq C\left( \int_{B_{2r_k e^{-s_k}}\setminus B_{r_ke^{-s_k}/2}} |\g \vn_{\Phi}|^2_{g_\Phi} d\vol_{g_{\Phi}} \right)^2.
	\end{align*}
	Since $\Phi$ is conformal and the Dirichlet energy is conformally invariant, we obtain
	\begin{align*}
		|\g (M_k Y_k)(s_k,\vp_k)|^4_\xi &\leq C\left( \int_{B_{2r_k e^{-s_k}}\setminus B_{r_ke^{-s_k}/2}} |\g \vn_{\Phi}|^2 dz \right)^2.
	\end{align*}
	By \cite[Proposition 1.2]{bernard2013a}, it holds $\g \vn_\Phi \in L^p(B_1)$, for any $p>1$. Thus, if $z_k := r_k e^{-s_k}$, we obtain by \eqref{eq:estimate_Q}:
	\begin{align*}
		|z_k|^4\|\scal{\dr^2_{zz} Y}{\dr^2_{zz} Y}_\eta \|_{L^\infty\left( B_{2|z_k|}\setminus B_{\frac{|z_k|}{2}} \right)} &\leq C\left( |z_k|^{2-\frac{2}{p}} \left\| |\g \vn_\Phi|^2 \right\|_{L^p\left( B_{2|z_k|}\setminus B_{\frac{|z_k|}{2}} \right)} \right)^2.
	\end{align*}
	Chosing $p=4$, we obtain $2-\frac{2}{p} = \frac{3}{2}$. Hence, the above estimate becomes:
	\begin{align*}
		|z_k| \|\scal{\dr^2_{zz} Y}{\dr^2_{zz} Y}_\eta \|_{L^\infty \left( B_{2|z_k|}\setminus B_{\frac{|z_k|}{2}} \right)} &\leq C \left\| \g \vn_\Phi \right\|_{L^8 \left(B_{2|z_k|}\setminus B_{\frac{|z_k|}{2}} \right)}^4.
	\end{align*}
	Thus, we end up with the following asymptotic expansion:
	\begin{align*}
		|z_k| \|\scal{\dr^2_{zz} Y}{\dr^2_{zz} Y}_\eta \|_{L^\infty\left( B_{2|z_k|}\setminus B_{\frac{|z_k|}{2}} \right)} \ust{k\to\infty}{=}o(1).
	\end{align*}
	Hence, $\Qr$ has no pole of order 1 or more: $\Qr$ is bounded across the singularity.

	\bibliographystyle{plain}
	\bibliography{branched_sphere}

\begin{thebibliography}{10}

\bibitem{bernard2023}
Yann Bernard, Paul Laurain, and Nicolas Marque.
\newblock Energy {{Estimates}} for the {{Tracefree Curvature}} of {{Willmore
  Surfaces}} and {{Applications}}.
\newblock {\em Archive for Rational Mechanics and Analysis}, 247(1):8, January
  2023.

\bibitem{bernard2013}
Yann Bernard and Tristan Rivi{\`e}re.
\newblock Asymptotic {{Analysis}} of {{Branched Willmore Surfaces}}.
\newblock {\em Pacific Journal of Mathematics}, 265(2):257--311, August 2013.

\bibitem{bernard2013a}
Yann Bernard and Tristan Rivi{\`e}re.
\newblock Singularity removability at branch points for {{Willmore}} surfaces.
\newblock {\em Pacific Journal of Mathematics}, 265(2):257--311, 2013.

\bibitem{bernard2014}
Yann Bernard and Tristan Rivi{\`e}re.
\newblock Energy quantization for {{Willmore}} surfaces and applications.
\newblock {\em Annals of Mathematics. Second Series}, 180(1):87--136, 2014.

\bibitem{blatt2009}
Simon Blatt.
\newblock A singular example for the {{Willmore}} flow.
\newblock {\em Analysis}, 29:407--430, January 2009.

\bibitem{bryant1984}
Robert~L. Bryant.
\newblock A duality theorem for {{Willmore}} surfaces.
\newblock {\em Journal of Differential Geometry}, 20(1):23--53, 1984.

\bibitem{chen2014}
Jingyi Chen and Yuxiang Li.
\newblock Bubble tree of branched conformal immersions and applications to the
  {W}illmore functional.
\newblock {\em Amer. J. Math.}, 136(4):1107--1154, 2014.

\bibitem{helein2002}
Fr{\'e}d{\'e}ric H{\'e}lein.
\newblock {\em Harmonic Maps, Conservation Laws and Moving Frames}, volume 150
  of {\em Cambridge {{Tracts}} in {{Mathematics}}}.
\newblock {Cambridge University Press, Cambridge}, second edition, 2002.

\bibitem{hertrich-jeromin2003}
Udo {Hertrich-Jeromin}.
\newblock {\em Introduction to {{M\"obius}} Differential Geometry}, volume 300
  of {\em London {{Mathematical Society Lecture Note Series}}}.
\newblock {Cambridge University Press, Cambridge}, 2003.

\bibitem{hirsch2021}
Jonas Hirsch, Rob Kusner, and Elena {M{\"a}der-Baumdicker}.
\newblock Geometry of complete minimal surfaces at infinity and the
  {{Willmore}} index of their inversions, November 2021.

\bibitem{hirsch2023}
Jonas Hirsch and Elena {M{\"a}der-Baumdicker}.
\newblock On the {{Index}} of {{Willmore}} spheres.
\newblock {\em Journal of Differential Geometry}, 124(1):37--79, 2023.

\bibitem{huber1957}
Alfred Huber.
\newblock On subharmonic functions and differential geometry in the large.
\newblock {\em Commentarii Mathematici Helvetici}, 32:13--72, 1957.

\bibitem{kuwert2012b}
Ernst Kuwert and Yuxiang Li.
\newblock {$W^{2,2}$}-conformal immersions of a closed {{Riemann}} surface into
  {$\mathbb{R}^n$}.
\newblock {\em Communications in Analysis and Geometry}, 20(2):313--340, 2012.

\bibitem{kuwert2004}
Ernst Kuwert and Reiner Sch{\"a}tzle.
\newblock Removability of point singularities of {{Willmore}} surfaces.
\newblock {\em Annals of Mathematics}, 160(1):315--357, July 2004.

\bibitem{kuwert2007}
Ernst Kuwert and Reiner Sch{\"a}tzle.
\newblock Branch points of {{Willmore}} surfaces.
\newblock {\em Duke Mathematical Journal}, 138(2):179--201, June 2007.

\bibitem{kuwert2012}
Ernst Kuwert and Reiner Sch{\"a}tzle.
\newblock The {{Willmore}} functional.
\newblock In Giuseppe Mingione, editor, {\em Topics in {{Modern Regularity
  Theory}}}, {{CRM Series}}, pages 1--115, {Pisa}, 2012. {Edizioni della
  Normale}.

\bibitem{lamm2015}
Tobias Lamm and Huy~The Nguyen.
\newblock Branched {{Willmore}} spheres.
\newblock {\em Journal f\"ur die reine und angewandte Mathematik (Crelles
  Journal)}, 2015(701):169--194, April 2015.

\bibitem{laurain2018a}
Paul Laurain and Tristan Rivi{\`e}re.
\newblock Energy quantization of {{Willmore}} surfaces at the boundary of the
  moduli space.
\newblock {\em Duke Mathematical Journal}, 167(11):2073--2124, August 2018.

\bibitem{marque2021}
Nicolas Marque.
\newblock Conformal {{Gauss}} map geometry and application to {{Willmore}}
  surfaces in model spaces.
\newblock {\em Potential Analysis. An International Journal Devoted to the
  Interactions between Potential Theory, Probability Theory, Geometry and
  Functional Analysis}, 54(2):227--271, 2021.

\bibitem{martino2023}
Dorian Martino.
\newblock Energy quantization for {{Willmore}} surfaces with bounded index, May
  2023.

\bibitem{michelat2020}
Alexis Michelat.
\newblock On the {{Morse}} index of {{Willmore}} spheres in {$\mathbb{S}^3$}.
\newblock {\em Communications in Analysis and Geometry}, 28(6):1337--1406,
  2020.

\bibitem{michelat2021}
Alexis Michelat.
\newblock On the {{Morse}} index of branched {{Willmore}} spheres in 3-space.
\newblock {\em Calculus of Variations and Partial Differential Equations},
  60(4):126, June 2021.

\bibitem{michelat2022}
Alexis Michelat and Tristan Rivi{\`e}re.
\newblock The {{Classification}} of {{Branched Willmore Spheres}} in the
  3-{{Sphere}} and the 4-{{Sphere}}.
\newblock {\em Annales scientifiques de l'\'Ecole Normale Sup\'erieure},
  55(5):1199--1288, 2022.

\bibitem{muller1995}
S.~M{\"u}ller and V.~{\v S}ver{\'a}k.
\newblock On surfaces of finite total curvature.
\newblock {\em Journal of Differential Geometry}, 42(2):229--258, 1995.

\bibitem{nguyen2012}
Huy~The Nguyen.
\newblock Geometric rigidity for analytic estimates of {M}\"uller-\v sver\'ak.
\newblock {\em Math. Z.}, 272(3-4):1059--1074, 2012.

\bibitem{palmurella2022}
Francesco Palmurella and Tristan Rivi{\`e}re.
\newblock The {{Parametric Approach}} to the {{Willmore Flow}}.
\newblock {\em Advances in Mathematics}, 400:108257, May 2022.

\bibitem{riviere2008}
Tristan Rivi{\`e}re.
\newblock Analysis aspects of {{Willmore}} surfaces.
\newblock {\em Inventiones mathematicae}, 174(1):1--45, October 2008.

\bibitem{riviere2013b}
Tristan Rivi{\`e}re.
\newblock Lipschitz conformal immersions from degenerating {{Riemann}} surfaces
  with {$L^2$}-bounded second fundamental forms.
\newblock {\em Advances in Calculus of Variations}, 6(1):1--31, 2013.

\bibitem{riviere2014}
Tristan Rivi{\`e}re.
\newblock Variational principles for immersed surfaces with {$L^2$}-bounded
  second fundamental form.
\newblock {\em Journal f\"ur die Reine und Angewandte Mathematik. [Crelle's
  Journal]}, 695:41--98, 2014.

\bibitem{riviere2016a}
Tristan Rivi{\'e}re.
\newblock Weak immersions of surfaces with {$L^2$}-bounded second fundamental
  form.
\newblock In {\em Geometric Analysis}, volume~22 of {\em {{IAS}}/{{Park City
  Math}}. {{Ser}}.}, pages 303--384. {Amer. Math. Soc., Providence, RI}, 2016.

\bibitem{riviere2020}
Tristan Rivi{\`e}re.
\newblock Willmore minmax surfaces and the cost of the sphere eversion.
\newblock {\em Journal of the European Mathematical Society}, 23(2):349--423,
  October 2020.

\bibitem{saearp2015}
Ricardo Sa~Earp and Eric Toubiana.
\newblock {\em Topologie, Courbure et Structure Conforme Sur Les Surfaces}.
\newblock December 2015.

\end{thebibliography}
\end{document}